\documentclass[reqno]{amsart}
\usepackage{amsmath}
\usepackage{amssymb}
\usepackage{amsthm}
\usepackage{cases}
\usepackage{mathrsfs}
\usepackage{hyperref}

\title{Andrews-Gordon type series for the level 5 and 7 standard modules of 
the affine Lie algebra $A^{(2)}_2$}

\author{Motoki Takigiku}
\address{Graduate School of Natural Science and Technology, Okayama University, Okayama 700-8530, Japan}
\email{takigiku@math.okayama-u.ac.jp}

\author{Shunsuke Tsuchioka}
\address{Department of Mathematical and Computing Sciences, Tokyo Institute of Technology, Tokyo 152-8551, Japan}
\email{tshun@kurims.kyoto-u.ac.jp}

\date{Jun 4, 2020}
\keywords{integer partitions,
Rogers-Ramanujan identities,
affine Lie algebras,
vertex operators,
Andrews-Gordon identities, $q$-series, hypergeometric series}
\subjclass[2010]{Primary~11P84, Secondary~05E10}

\usepackage{tikz}

\def\node#1#2{\overset{#1}{\underset{#2}{\circ}}}

\def\ver#1#2{\overset{{\llap{$\scriptstyle#1$}\displaystyle\circ{\rlap{$\scriptstyle#2$}}}}{\scriptstyle\vert}}

\usetikzlibrary{arrows}
\tikzstyle{every picture}+=[remember picture]
\tikzstyle{na} = [baseline=-.5ex]
\tikzstyle{mine}= [arrows={angle 90}-{angle 90},thick]

\def\Llleftarrow{%
\lower2pt\hbox{\begingroup
\tikz
\draw[shorten >=0pt,shorten <=0pt] (0,3pt) -- ++(-1em,0) (0,1pt) -- ++(-1em-1pt,0) (0,-1pt) -- ++(-1em-1pt,0) (0,-3pt) -- ++(-1em,0) (-1em+1pt,5pt) to[out=-105,in=45] (-1em-2pt,0) to[out=-45,in=105] (-1em+1pt,-5pt);
\endgroup}
}

\def\Rrrightarrow{\vcenter{\hbox{\rotatebox{180}{\Llleftarrow}}}}

\newtheorem{Thm}{Theorem}[section]

\newtheorem{Conj}[Thm]{Conjecture}

\newtheorem{Lem}[Thm]{Lemma}

\newtheorem{Rem}[Thm]{Remark}

\newcommand{\Z}{\mathbb{Z}}

\newcommand{\princhar}[2]{\chi_{#1}(#2)}
\newcommand{\sums}[1]{\sum_{\substack{#1}}}

\newcommand{\qp}[2][\infty]{(q^{#2};q^{#2})_{#1}}
\newcommand{\qpp}[3][\infty]{(q^{#2};q^{#3})_{#1}}

\newcommand{\eqEuler}{A}
\newcommand{\eqEulerr}{B}
\newcommand{\eqbinom}{C}

\newcommand{\gnahmsp}[3]{\sum_{#3} \frac{#2}{#1}}
\newcommand{\gnahms}[4]{\sum_{#4} (-1)^{#2} \frac{#3}{#1}}

\newcommand{\GEE}{\mathfrak{g}}

\newcommand{\quadd}[2]{{#1\binom{i}{2} + #2\binom{j}{2}}}
\newcommand{\intee}[1]{{#1 ij}}
\newcommand{\linee}[2]{{#1 i + #2 j}}

\newcommand{\denommm}[3]{\qp[i]{#1} \qp[j]{#2} \qp[k]{#3}}
\newcommand{\quaddd}[3]{{
    #1\binom{i}{2} + #2\binom{j}{2}  + #3\binom{k}{2} 
}}
\newcommand{\inteee}[3]{{
    #1 ij + #2 ik + #3 jk
}}
\newcommand{\lineee}[3]{{
    #1 i + #2 j + #3 k
}}

\newcommand{\denommmm}[4]{\qp[i]{#1} \qp[j]{#2} \qp[k]{#3} \qp[\ell]{#4}}
\newcommand{\quadddd}[4]{{
    #1\binom{i}{2} + #2\binom{j}{2} + #3\binom{k}{2} + #4\binom{\ell}{2}
}}
\newcommand{\inteeee}[6]{{
    #1 ij + #2 ik + #3 i\ell +
    #4 jk + #5 j\ell + 
    #6 k\ell
}}
\newcommand{\lineeee}[4]{{
    #1 i + #2 j + #3 k + #4 \ell
}}

\newcommand{\numerrr}[9]{{
    #1\binom{i}{2} + #2\binom{j}{2}  + #3\binom{k}{2} 
    + #4 ij + #5 ik + #6 jk
    + #7 i + #8 j + #9 k
}}

\newcommand{\qnumerrr}[9]{q^{\numerrr{#1}{#2}{#3}{#4}{#5}{#6}{#7}{#8}{#9}}}
\newcommand{\numANiJusanIchi}{1}
\newcommand{\numANiJusanNi}{2}
\newcommand{\numANiJusanSan}{3}

\usepackage{graphicx}

\begin{document}
\maketitle
\maketitle

\begin{abstract}
We give Andrews-Gordon type series for
the principal characters of 
the level 5 and 7 standard modules
of the affine Lie algebra $A^{(2)}_{2}$.
We also give conjectural series for some level 2 modules of $A^{(2)}_{13}$.
\end{abstract}

\section{Introduction}
In this paper, we use the $q$-Pochhammer symbol: for $n\in\mathbb{N},m\in\mathbb{N}\cup\{\infty\}$,
\begin{align*}
(x;q)_\infty := \prod_{i\ge 0} (1-xq^i),\quad
(x;q)_n := \prod_{i=0}^{n-1} (1-xq^i),\quad
(a_1,\dots,a_k;q)_{m}:= (a_1;q)_m \cdots (a_k;q)_m.
\end{align*}

\subsection{The Andrew-Gordon identities}\label{RRintro}
The \emph{Rogers-Ramanujan identities}
\begin{align}
\sum_{n\geq 0}\frac{q^{n^2}}{(q;q)_n} = \frac{1}{(q,q^4;q^{5})_{\infty}}, \quad
\sum_{n\geq 0}\frac{q^{n^2+n}}{(q;q)_n} = \frac{1}{(q^2,q^3;q^{5})_{\infty}}
\label{RRidentities}
\end{align}
was one of the motivations for inventing the vertex operators~\cite[\S14]{Kac} in the theory of affine Lie algebras 
(see ~\cite{Lep}). It started from Lepowsky-Milne's observation~\cite{LM}:
\begin{align*}
\chi_{A^{(1)}_1}(2\Lambda_0+\Lambda_1) = \frac{1}{(q,q^4;q^{5})_{\infty}}, \quad
\chi_{A^{(1)}_1}(3\Lambda_0) = \frac{1}{(q^2,q^3;q^{5})_{\infty}}.
\end{align*}
Here, $\chi_A(\lambda)$ (called the \emph{principal chacater}) stands 
for the principally specialized character of the vacuum space $\Omega(V(\lambda))$ ~\cite[\S7]{Fil}
for the integrable highest weight module (a.k.a. the standard module) $V(\lambda)$
associated with a dominant integral weight $\lambda\in P^+$
of the affine Lie algebra $\GEE(A)$.
We obey the numbering of vertices of the affine Dynkin diagram $A$ in ~\cite[\S4]{Kac} and
duplicate $A^{(1)}_1,A^{(2)}_2,A^{(2)}_{\textrm{odd}}$ as Figure \ref{twisted}.
The \emph{level} of $\sum_{i\in I}d_i\Lambda_i\in P^+$ is given by $\sum_{i\in I}\check{a}_id_i$,
where the colabel $\check{a}_i$ is the number written on the vertex $\alpha_i$
in the figure.
We can expand $\chi_A(\lambda)$ into an explicit infinite product via Lepowsky's numerator formula (see ~\cite{Bos}).

After the success of vertex operator theoretic proofs of
the Rogers-Ramanujan identities~\cite{LW1,LW2,LW3}, it has been expected that, for each $A$ and $\lambda$, 
there should exist ``Rogers-Ramanujan type identities'' 
whose infinite products are given by $\chi_A(\lambda)$. 

The \emph{Andrews-Gordon identities} (Theorem \ref{AGTHM})
can be seen as an instance of this expectation
because of an existence of a vertex operator theoretic proof for it~\cite{LP}. 
Note that the infinite product (the right hand side) in Theorem \ref{AGTHM} is equal to $\chi_{A^{(1)}_1}((2k-i)\Lambda_0+(i-1)\Lambda_1)$.
This is the case of level $2k-1$ and the Rogers-Ramanujan identities \eqref{RRidentities} are the cases when $k=2$ and $i=2,1$. An even level analog is known as the Andrews-Bressoud identities (see ~\cite[\S3.2.2]{Sil}).

\begin{Thm}[{\cite{An3}}]\label{AGTHM}
Let $1\leq i\leq k$. Putting $N_j=n_j+\cdots+n_{k-1}$, 
we have
\begin{align*}
\sum_{n_1,\cdots,n_{k-1}\geq 0} \frac{q^{N_1^2+\cdots+N_{k-1}^2+N_i+\cdots+N_{k-1}}}{(q;q)_{n_1}\cdots (q;q)_{n_{k-1}}}
= \frac{(q^i,q^{2k+1-i},q^{2k+1};q^{2k+1})_{\infty}}{(q;q)_{\infty}}.
\end{align*}
\end{Thm}

As do the Rogers-Ramanujan identities, 
Theorem \ref{AGTHM} has
an intepretation as a partition theorem~\cite{Gor}.
A \emph{partition theorem} is a statement of the form ``For any $n\geq 0$, partitions $(\lambda_1,\cdots,\lambda_{\ell})$ of $n$ with
condtion $C$ are equinumerous to partitions of $n$ with condition $D$''.
Theorem \ref{AGTHM} is equivalent to the partition theorem, where
\begin{enumerate}
\item[$C$:] $1\leq\forall j\leq \ell-k+1,\lambda_j-\lambda_{j+k-1}\geq 2$ and 
$|\{1\leq j\leq\ell\mid \lambda_j=1\}|<i$,
\item[$D$:] $1\leq\forall j\leq \ell,\lambda_j\not\equiv 0,\pm i\pmod{2k+1}$.
\end{enumerate}

\begin{figure}
\begin{align*}
\begin{array}{r@{\quad}l@{\qquad}l@{\quad}l@{\qquad}l@{\quad}l}
A_1^{(1)} & \node{1}{\alpha_0} \Leftrightarrow \node{1}{\alpha_1} & 
A_2^{(2)} &\node{2}{\alpha_0} {\quad\!\!\!\!\!\!\!\!}\Rrrightarrow{\quad\!\!\!\!\!} \node{1}{\alpha_1} &
A_{2\ell-1}^{(2)}  &\node{1}{\alpha_1}-\node{\ver{1}{\alpha_0}}{\alpha_2}\!\!{}^2-\node{2}{\alpha_3}-\cdots-\node{2}{\alpha_{\ell-1}}\Leftarrow\node{2}{\alpha_\ell} 
\end{array}
\end{align*}
\caption{The affine Dynkin diagrams $A^{(1)}_1,A^{(2)}_{2},A^{(2)}_{\textrm{odd}}$.}
\label{twisted}
\end{figure}

\subsection{The main theorems}\label{mainres}
In this paper (see also ~\cite[Conjecture 1.1]{CL}), 
``Andrews-Gordon type series'' stands for an infinite sum of the form 
\begin{align*}
\sum_{i_1,\cdots,i_s\geq 0}
\frac{(-1)^{\sum_{\ell=1}^{s}L_\ell i_\ell}q^{\sum_{\ell=1}^{s}a_{\ell\ell}{\ell\choose 2}+\sum_{1\leq j<k\leq s}a_{jk}i_ji_k+\sum_{\ell=1}^{s}B_\ell i_\ell}}
{(q^{C_1};q^{C_1})_{i_1}\cdots(q^{C_s};q^{C_s})_{i_s}},
\end{align*}
for some $(L_\ell)_{\ell=1}^{s},(B_\ell)_{\ell=1}^{s}\in\mathbb{Z}^s,(C_\ell)_{\ell=1}^{s}\in\mathbb{N}_{\geq 1}^s$ and $(a_{jk})_{1\leq j\leq k\leq s}\in\Z^{{s+1 \choose 2}}$.
Usually, we would like to impose some non-degeneracy conditions such as $a_{jk}>0$ for all $1\leq j\leq k\leq s$
in order not to make the problem trivial.
In the rest, we put 
\begin{align*}
[x;q]_m=(x,q/x;q)_m,\quad
[a_1,\cdots, a_k;q]_m=[a_1;q]_m\cdots [a_k;q]_m
\end{align*}
for $m\in\mathbb{N}\cup\{\infty\}$. 
The purpose of this paper is to give the following.

\begin{Thm}\label{theo:A22:l5}
Concerning level 5 standard modules for $A^{(2)}_2$, we have
\begin{align*}
\sum_{i,j,k\ge0} (-1)^k \frac{\qnumerrr{}{8}{2}{2}{2}{4}{}{5}{}}{\denommm{}{2}{2}}
= \frac{1}{[q^2,q^3,q^4,q^5;q^{16}]_\infty}(=\princhar{A_2^{(2)}}{5\Lambda_0}), \\ 
\sum_{i,j,k\ge0} (-1)^k \frac{\qnumerrr{}{8}{2}{2}{2}{4}{}{7}{3}}{\denommm{}{2}{2}}
= \frac{1}{[q,q^4,q^6,q^7;q^{16}]_\infty}(=\princhar{A_2^{(2)}}{\Lambda_0+2\Lambda_1}). 
\end{align*}
\end{Thm}

\begin{Thm}\label{theo:A22:7:1}
Concerning level 7 standard modules for $A^{(2)}_2$, we have
\begin{align*}
\sum_{i,j,k\ge0}\frac{q^\numerrr{}{8}{10}{2}{2}{8}{}{4}{5}}{\qp[i]{}\qp[j]{2}\qp[k]{2}}
= \frac{1}{[q,q^3,q^4,q^5,q^7,q^9;q^{20}]_\infty}(=\princhar{A^{(2)}_{2}}{5\Lambda_0+\Lambda_1}), \\ 
\sum_{i,j,k\ge0}\frac{q^\numerrr{}{8}{10}{2}{2}{8}{}{8}{9}}{\qp[i]{}\qp[j]{2}\qp[k]{2}}
= \frac{1}{[q,q^3,q^5,q^7,q^8,q^9;q^{20}]_\infty}(=\princhar{A^{(2)}_{2}}{\Lambda_0+3\Lambda_1}). 
\end{align*}
\end{Thm}

\begin{Thm}\label{theo:A22:7:2}
Concerning level 7 standard modules for $A^{(2)}_2$, we have
\begin{align*}
  {} &{} \sum_{i,j,k,\ell\ge0} (-1)^k \frac{q^{\quadddd{}{2}{2}{8} + \inteeee{}{}{2}{4}{4}{4} + \lineeee{}{3}{}{6} } }{\qp[i]{} \qp[j]{2} \qp[k]{2} \qp[l]{4}}\\
  &= \frac{1}{[q^2,q^3,q^4,q^5,q^6,q^7;q^{20}]_\infty}(=\princhar{A^{(2)}_{2}}{7\Lambda_0}), \\
  {} &{} \sum_{i,j,k,\ell\ge0} (-1)^k \frac{q^{\quadddd{}{2}{2}{8} + \inteeee{}{}{2}{4}{4}{4} + \lineeee{}{2}{4}{8} } }{\qp[i]{} \qp[j]{2} \qp[k]{2} \qp[l]{4}}\\
  &= \frac{1}{[q,q^2,q^5,q^6,q^8,q^9;q^{20}]_\infty}(=\princhar{A^{(2)}_{2}}{3\Lambda_0+2\Lambda_1}).
\end{align*}
\end{Thm}

Note that the level 5 module $V(3\Lambda_0+\Lambda_1)$ is missing in Theorem \ref{theo:A22:l5}, 
but
\begin{align*}
\princhar{A_2^{(2)}}{3\Lambda_0+\Lambda_1} = \frac{1}{[q,q^3,q^5,q^7;q^{16}]_\infty} = \frac{1}{(q;q^2)_\infty},
\end{align*}
can be written as a form of Andrews-Gordon type series (e.g. put $x=q$ in \S\ref{prep} (\eqEulerr)).

\subsection{Comments on the proofs}
The following steps are common
in proving that a multisum is equal to an infinite product.
\begin{enumerate}
\item[(S1)] Reduce the multisum to a single sum.
\item[(S2)] Search for lists of identities which deduce the desired result.
\end{enumerate}

The step (S1) for Theorem \ref{theo:A22:l5}, \ref{theo:A22:7:1}, \ref{theo:A22:7:2}
uses a similar technique to the proof of Kanade-Russell's conjectures of modulo 12~\cite[$I_5,I_6$]{KR}
by Bringmann et.al.~\cite[\S4.9, \S4.10]{BJM}.
For (S2), we employ Slater's list~\cite{Sla} (see also Remark \ref{otherslater}).
\begin{Thm}[{\cite[(39)=(83),(38)=(86),(99),(94)]{Sla}}]\label{SLATERTHM}
\begin{alignat*}{3}
\sum_{n\ge0} \frac{q^{2n^2}}{(q;q)_{2n}} &= \princhar{A_2^{(2)}}{5\Lambda_0}, & 
{\quad} & \sum_{n\ge0} \frac{q^{2n^2+2n}}{(q;q)_{2n+1}} &= \princhar{A_2^{(2)}}{\Lambda_0+2\Lambda_1}, \\
\sum_{n\ge0} \frac{q^{n^2+n}}{(q;q)_{2n}} &= \princhar{A^{(2)}_{2}}{7\Lambda_0}, & 
{\quad} & \sum_{n\ge0} \frac{q^{n^2+n}}{(q;q)_{2n+1}} &= \princhar{A^{(2)}_{2}}{3\Lambda_0+2\Lambda_1}.
\end{alignat*}
\end{Thm}

\subsection{Toward $A^{(2)}_2$ analog of the Andrews-Gordon identities}\label{Atwotwo}
Let us recall the previous studies for $A^{(2)}_2$. 
For the level 2 case, the principal characters are obtained by inflating $q$ to $q^2$ from the
infinite products in \eqref{RRidentities}. 
For the level 3 (resp. level 4) case, the vacuum spaces are studied in ~\cite{Cap0} (resp. in ~\cite{Nan}), 
which resulted in conjectural partition theorems (see ~\cite[Theorem 5.2, Theorem 5.3, Conjecture 5.5, Conjecture 5.6, Conjecture 5.7]{Sil}) that were later proved in ~\cite{An4,Cap1,TX} (resp. ~\cite{TT}). 
The Andrews-Gordon type series are known as ~\cite[Corollary 18]{Ku2} (resp. ~\cite{TT}), which
are duplicated in Theorem \ref{KUres} (resp. Remark \ref{recstr}). For the level 3 case, see also \S\ref{Capsum}.

For the level 5 and 7 cases, the infinite products in 
Theorem \ref{theo:A22:l5}, Theorem \ref{theo:A22:7:1}, Theorem \ref{theo:A22:7:2}
appear in 
~\cite[Theorem 4, Theorem 3]{Hir}, ~\cite[Theorem 1]{Hir} and ~\cite[Theorem 1.7]{Cap2}, 
~\cite[Theorem 1.6]{Cap2} and ~\cite[Theorem 2]{Hir} respectively.
Those partition theorems look quite different from those for the aforementioned level 3 and 4 cases~\cite{Cap0,Nan}.
It is natural to investigate partition theorems via the vertex operators (other than ~\cite{Hir,Cap2})
that are related with our identities (like Proposition \ref{KUresss}).
As far as we know, almost nothing is known on the level 6 case except ~\cite[(1.3)--(1.6)]{MS} (see also ~\cite{Sil2}).

\subsection{Toward $A^{(2)}_{\textrm{odd}}$, level 2 analog of the Andrews-Gordon identities}
Recently, related to the expectation mentioned in \S\ref{RRintro}, 
some (conjectural) Andrews-Gordon type series for other affine Dynkin diagrams were found.
A famous example is Kur\c{s}ung\"oz's reformulation~\cite[Conjecture 6.1]{Kur} 
of Kanade-Russell conjectures of modulo 9~\cite{KR},
which are regarded as the level 3 identities of $D^{(3)}_4$ (see ~\cite[\S5.2]{Sil}).

Another actively studied levels and Lie types are level 2 of $A^{(2)}_{2\ell+1}$ (see ~\cite[\S1.1]{KR2}).
When $\ell=2$, 
Andrews-Gordon type series for $\chi_{A^{(2)}_5}(\Lambda_0+\Lambda_1),\chi_{A^{(2)}_5}(\Lambda_3)$ are found as in ~\cite[(21),(22)]{Ku2}. As shown by ~\cite{Kan}, they are related with the classical G\"ollnitz partition 
theorems~\cite[Theorem 2.42, Theorem 2.43]{Sil}.
When $\ell=3$, $\chi_{A^{(2)}_7}(\Lambda_0+\Lambda_1),\chi_{A^{(2)}_7}(\Lambda_3)$
are the infinite products in \eqref{RRidentities}. 
When $\ell=4$, Bringmann et.al.~\cite[\S3, \S4.1, \S4.2]{BJM} 
proved
Andrews-Gordon type series for
$\chi_{A^{(2)}_9}(\Lambda_0+\Lambda_1),\chi_{A^{(2)}_9}(\Lambda_3),\chi_{A^{(2)}_9}(\Lambda_5)$
which were conjectured in ~\cite[(3.2),(3.4),(3.6)]{KR2} together with
interpretations as partition theorems. See also ~\cite{Ros}.

Noting the principal characters for level 4 modules of $A^{(2)}_2$ coincide with
some of those for level 2 of $A^{(2)}_{11}$ (see Remark \ref{recstr}), 
in Conjecture \ref{conj:A13:3sum:1}, Conjecture \ref{conj:A13:3sum:2}, Conjecture \ref{conj:A13:4sum:1})
we give conjectural Andrews-Gordon type series for 
\begin{align*}
\chi_{A^{(2)}_{13}}(\Lambda_0+\Lambda_1),\quad
\chi_{A^{(2)}_{13}}(\Lambda_3),\quad
\chi_{A^{(2)}_{13}}(\Lambda_5),\quad
\chi_{A^{(2)}_{13}}(\Lambda_7).
\end{align*}

It would be interesting if one can find a pattern in Andrews-Gordon type series obtained so far.
For example, some of them for level 4 and 5 modules of $A^{(2)}_2$ (resp. level 2 of $A^{(2)}_{11}$
and $A^{(2)}_{13}$) share a recursive structure as in
Remark \ref{recstr} and Remark \ref{recstrr} that we can also observe in
the original (Theorem \ref{AGTHM}): namely, 
the Andrew-Gordon series for
$\chi_{A^{(1)}_1}((2(k-1)-i)\Lambda_0+(i-1)\Lambda_1)$
is obtained by deleting $n_{k-1}$ (or substituting $n_{k-1}=0$)
in that for $\chi_{A^{(1)}_1}((2k-i)\Lambda_0+(i-1)\Lambda_1)$.

We hope this paper contributes to the expectation mentioned in \S\ref{RRintro}
which has been shared in the community since Lepowsky-Milne's observation.

\hspace{0mm}

\noindent{\bf Organization of the paper.} 
The paper is organized as follows.
In \S\ref{prep}, we show some auxilary summation formulas via Euler's identities, the $q$-binomial theorem and
the $q$-WZ method. Then, we prove Theorem \ref{theo:A22:l5}, Theorem \ref{theo:A22:7:1}, Theorem \ref{theo:A22:7:2}
in \S\ref{pr:A22:l5}, \S\ref{pr:A22:l7:1}, \S\ref{pr:A22:l7:2} respectively.
In \S\ref{conjs} (resp. \S\ref{Capsum}), we discuss some Andrews--Gordon type series related to level 2 (resp. level 3) modules of $A^{(2)}_{13}$ (resp. $A^{(2)}_2$).

\hspace{0mm}

\noindent{\bf Acknowledgments.} 
We thank S.~Kanade and M.~Russell for helpful
discussions.
This work was supported by the Research Institute for Mathematical
Sciences, an International Joint Usage/Research Center located in Kyoto
University and the TSUBAME3.0 supercomputer at Tokyo Institute of Technology.
S.T.\,was supported in part by JSPS Kakenhi Grants 17K14154, 20K03506 and by Leading Initiative for Excellent Young Researchers, MEXT, Japan.
M.T.\,was supported in part by Start-up research support from Okayama University.

\section{Preparations}\label{prep}
Recall the \emph{Euler's identities} and the \emph{$q$-binomial theorem}~\cite[(II.1),(II,2),(II.3)]{GR}.
\begin{align*}
  \sum_{n\ge 0} \frac{x^n}{(q;q)_n} \stackrel{(\eqEuler)}{=} \frac{1}{(x;q)_\infty},
  \sum_{n\ge 0} \frac{q^{\binom{n}{2}} x^n}{(q;q)_n} \stackrel{(\eqEulerr)}{=} (-x;q)_\infty,
  \sum_{n\ge 0} \frac{(a;q)_n}{(q;q)_n} x^n \stackrel{(\eqbinom)}{=} \frac{(ax;q)_\infty}{(x;q)_\infty}.
\end{align*}

\begin{Lem}\label{AUXLEM}
For any $M\in\Z_{\ge 0}$, we have
\begin{enumerate}
\item\label{eq:ms:fold:22:00:0,0,0,0,1:11} $\displaystyle\gnahmsp{\qp[i]{2} \qp[j]{2}} {q^{j}}   {\substack{i,j\geq 0 \\ i+j=M}} = \frac{1}{\qp[M]{}}$,
\item\label{eq:ms:fold:2,2:1,0:2,0,0,1,0:1,1} $\displaystyle\gnahms{\qp[i]{2} \qp[j]{2}}    {j} {q^{2\binom{j}{2}+j}}   {\substack{i,j\geq 0 \\ i+j=M}} = \frac{\qpp[M]{}{2}}{\qp[M]{2}}$,
\item\label{eq:ms:fold:112:010:0,0,0;1,0,0;0,1,1:112} $\displaystyle\gnahms{\qp[i]{} \qp[j]{} \qp[k]{2}}    {j} {q^{ij+j+k}}   {\substack{i,j\geq 0 \\ i+j+2k=M}} = \frac{(-q;q)_{[M/2]}}{(q;q)_{[M/2]}}$,
\item\label{eq:ms:fold:2,2:1,0:6,0,2,4,0:1,1} $\displaystyle\gnahms{\qp[i]{} \qp[j]{}}    {j}   {q^{3\binom{j}{2}+ij+2j}} {\substack{i,j\geq 0 \\ i+j=M}} = \frac{\qp[2M]{}}{\qp[M]{}^2}$.
\end{enumerate}
\end{Lem}

\begin{proof}
By (\eqEuler), we get \eqref{eq:ms:fold:22:00:0,0,0,0,1:11} as follows.
\begin{align*}
\sum_{i,j\geq 0} \frac{q^j}{\qp[i]{2}\qp[j]{2}}x^{i+j}
= \frac{1}{(x;q^2)_\infty}\frac{1}{(xq;q^2)_\infty} 
= \frac{1}{(x;q)_\infty}
= \sum_{M\geq 0} \frac{x^M}{\qp[M]{}}.
\end{align*}
  
Similarly, we get \eqref{eq:ms:fold:2,2:1,0:2,0,0,1,0:1,1} by (\eqEuler), (\eqEulerr), (\eqbinom) as follows.
\begin{align*}
\sum_{i,j} \frac{x^i}{\qp[i]{2}}  \frac{(-1)^j q^{2\binom{j}{2}+j} x^j}{\qp[j]{2}}
= \frac{1}{(x;q^2)_\infty} (xq;q^2)_\infty
= \sum_{M\geq 0} \frac{\qpp[M]{}{2}}{\qp[M]{2}} x^M.
\end{align*}
  
To prove \eqref{eq:ms:fold:112:010:0,0,0;1,0,0;0,1,1:112}, we calculate 
the generating series of both sides by (\eqEuler), (\eqEulerr). 
\begin{align*}
  {} &\sum_{i,j,k\geq 0} (-1)^j \frac{q^{ij+j+k}}{\qp[i]{} \qp[j]{} \qp[k]{2}} x^{i+j+2k}=
  \bigg(\sum_{j\geq 0} \big(\frac{(-xq)^j}{\qp[j]{}} \sum_{i\geq 0} \frac{(xq^{j})^i}{\qp[i]{}}\big)\bigg)
  \cdot
  \sum_{k\geq 0} \frac{(x^2q)^k}{\qp[k]{2}}\\
&=
  \bigg(\sum_{j\geq 0} \frac{(-xq)^j}{\qp[j]{}} 
  \frac{1}{(xq^j;q)_\infty}\bigg)
  \cdot
  \frac{1}{(x^2q;q^2)_\infty}
  = 
  \frac{1}{(x;q)_\infty}
  \frac{1}{(x^2q;q^2)_\infty}
  \sum_{j\geq 0} \frac{(x;q)_j}{\qp[j]{}} (-xq)^j.
  \end{align*}
Thus, the generating series of the left hand side is equal to $\displaystyle\frac{1}{(x;q)_\infty}\frac{1}{(x^2q;q^2)_\infty}\frac{(-x^2q;q)_\infty}{(-xq;q)_\infty}$ 
by (\eqbinom). We easily see that it is equal to the generating series of the right hand side:
\begin{align*}
  \sum_{M\geq 0} \frac{(-q;q)_{[M/2]}}{(q;q)_{[M/2]}} x^M=
  \sum_{N\geq 0} \frac{(-q;q)_{N}}{(q;q)_{N}} x^{2N} (1 + x)=
  (1+x) \frac{(-x^2q;q)_\infty}{(x^2;q)_\infty}.
\end{align*}
    
In the proof of \eqref{eq:ms:fold:2,2:1,0:6,0,2,4,0:1,1},
we promise $1/(q;q)_n=0$ if $n<0$. 
For $M,j\geq 0$, we let
\begin{align*}
  f_{M,j} := (-1)^{j} \frac{
    q^{3\binom{j}{2}+(M-j)j+2j} \qp[M]{}^2
  }{\qp[M-j]{} \qp[j]{} \qp[2M]{}}
\end{align*}
so that it suffices to prove (by the $q$-WZ method~\cite{Koe}) $\sum_{j\geq 0} f_{M,j} = 1$ (note that $f_{M,j}=0$ when $j>M$).
The $q$-Zeilberger algorithm helps us finding an expression
\begin{align*}
g_{M,j}
= (-1)^j
\frac{
  (1 - q^{M+1-j} - q^{2M+2-j})
  q^{3\binom{j}{2}+(M-j)j+2j} 
  \qp[M]{}^2
}{
  (1+q^{M+1}) (1-q^{2M+1}) 
  \qp[M-j+1]{} \qp[j-1]{} \qp[2M]{}
}
\end{align*}
for which we can verify that 
$g_{M,j}\neq0\Rightarrow 0\le j-1\le M$ and 
$f_{M+1,j} - f_{M,j} = g_{M,j+1} - g_{M,j}$ for any $M,j\ge 0$.
This implies $\sum_{j\geq 0} f_{M+1,j} - \sum_{j\geq 0} f_{M,j}=0$ for any $M\ge 0$.
Now we only need to see $\sum_{j\geq 0} f_{0,j}=1$, which is obvious.
\end{proof}

\section{Proof of Theorem \ref{theo:A22:l5}}\label{pr:A22:l5}
Note that the Andrews-Gordon type series is of the form
\begin{align}
\sum_{i,j,k\ge0} (-1)^k \frac{\qnumerrr{}{8}{2}{2}{2}{4}{}{(5+2a)}{(1+2a)}}{\denommm{}{2}{2}}
\label{eq:l5:prf:1}
\end{align}
for $a=0,1$. We rewrite the inner sum on $j$ and $k$ as
\begin{align*}
  {} &{} \sum_{j,k\ge0} (-1)^{k}
  \frac{q^{{8}\binom{j}{2} + {2}\binom{k}{2} + {4}jk + {(5+2a+2i)}j + ({1+2a+2i})k}}{\qp[j]{2} \qp[k]{2}}\\
  &= \sum_{M\ge0} 
  \sums{j,k\ge0 \\ j+k=M} (-1)^{j+M}
  \frac{q^{{6}\binom{j}{2} + {2}jk + 4j + 2\binom{M}{2} + (2a+2i+1)M}}{\qp[j]{2} \qp[k]{2}}.
\end{align*}
By Lemma \ref{AUXLEM} \eqref{eq:ms:fold:2,2:1,0:6,0,2,4,0:1,1} (with $q$ replaced by $q^2$), we see
\eqref{eq:l5:prf:1} is reduced to
\begin{align}
\sum_{i,M\ge0}(-1)^M \frac{(q^2;q^2)_{2M}}{(q;q)_i (q^2;q^2)_M^2}q^{\binom{i}{2}+M^2+2iM+i+2aM}.
\label{eq:l5:prf:2}
\end{align}
With (\eqEulerr), the inner sum on $i$ is rewritten as
\begin{align*}
\sum_{i\ge0} 
\frac{q^{\binom{i}{2}+(2M+1)i}}{(q;q)_i}
= (-q^{2M+1};q)_\infty
= \frac{(-q;q)_\infty}{(-q;q)_{2M}}
= \frac{1}{(q;q^2)_\infty} \frac{(q;q)_{2M}}{(q^2;q^2)_{2M}}.
\end{align*}
Hence,  by Lemma \ref{AUXLEM} \eqref{eq:ms:fold:2,2:1,0:2,0,0,1,0:1,1}, we see \eqref{eq:l5:prf:2}$\cdot(q;q^2)_\infty$ is equal to
\begin{align*}
\sum_{M\ge0} (-1)^M \frac{q^{M^2+2aM} (q;q^2)_M}{(q^2;q^2)_M}
= \sum_{m,n\ge0} (-1)^m \frac{q^{2\binom{m}{2}+4\binom{n}{2}+2mn+(1+2a)m+(2+2a)n}}{\qp[m]{2}\qp[n]{2}}.
\end{align*}
Finally we use (\eqEulerr) to rewrite the inner sum on $m$ as
\begin{align*}
\sum_{m\ge0}
(-1)^m
\frac{q^{2\binom{m}{2}+(1+2a+2n)m}}{\qp[m]{2}}
= (q^{1+2a+2n};q^2)_\infty
= \frac{(q;q^2)_\infty}{(q;q^2)_{n+a}}.
\end{align*}
Thus, we see $\displaystyle\eqref{eq:l5:prf:1}=\frac{\eqref{eq:l5:prf:2}\cdot(q;q^2)_\infty}{(q;q^2)_\infty}$
is equal to
\begin{align*}
\sum_{n\ge0} \frac{q^{4\binom{n}{2}+(2+2a)n}}{(q;q^2)_{n+a}(q^2;q^2)_n}
=\sum_{n\ge0} \frac{q^{2n^2+2an}}{(q;q)_{2n+a}},
\end{align*}
which proves Theorem \ref{theo:A22:l5} in virtue of Theorem \ref{SLATERTHM}.

\begin{Rem}\label{recstr}
  %
  In \cite[Theorem 2.2]{TT}
  the following identities 
  are shown:
  \begin{align}
    \sum_{i,k\ge0} (-1)^k \frac{q^{\binom{i}{2} + 2\binom{k}{2} + 2ik + i+k}}{\qp[i]{} \qp[k]{2}}
    &= \frac{1}{[q^2,q^3,q^4;q^{14}]_\infty}
    (=\princhar{A^{(2)}_{2}}{4\Lambda_0}=\princhar{A^{(2)}_{11}}{\Lambda_3}), 
    \label{eq:AG:234:14}
    \\
    \sum_{i,k\ge0} (-1)^k \frac{q^{\binom{i}{2} + 2\binom{k}{2} + 2ik + i+3k}}{\qp[i]{} \qp[k]{2}}
    &= \frac{1}{[q,q^4,q^6;q^{14}]_\infty}
    (=\princhar{A^{(2)}_{2}}{2\Lambda_0+\Lambda_1}=\princhar{A^{(2)}_{11}}{\Lambda_0+\Lambda_1}),
    \label{eq:AG:146:14}
    \\
    \sum_{i,k\ge0} (-1)^k \frac{q^{\binom{i}{2} + 2\binom{k}{2} + 2ik + 2i+3k}}{\qp[i]{} \qp[k]{2}}
    &= \frac{1}{[q^2,q^5,q^6;q^{14}]_\infty}
    (=\princhar{A^{(2)}_{2}}{2\Lambda_1}=\princhar{A^{(2)}_{11}}{\Lambda_5}).
    \label{eq:AG:256:14}
  \end{align}
We remark that \eqref{eq:AG:234:14} and \eqref{eq:AG:146:14} 
coincide with
double sums obtained by
taking the ``$j=0$ part'' of the triple sums in
Theorem \ref{theo:A22:l5}.
\end{Rem}

\section{Proof of Theorem \ref{theo:A22:7:1}}\label{pr:A22:l7:1}
Note that the Andrews-Gordon type series is of the form
\begin{align}\label{eq:l7:prf:1}
\sum_{i,j,k\ge0}\frac{q^{\quaddd{}{8}{10} + \inteee{2}{2}{8} + \lineee{}{(4+4a)}{(5+4a)} }}{\denommm{}{2}{2}}
\end{align}
for $a=0,1$. 
We rewrite the inner sum on $j$ and $k$ in \eqref{eq:l7:prf:1} as
\begin{align*}
\sums{j,k\ge0}\frac{q^{8\binom{j}{2} + 10\binom{k}{2} + 8jk + (2i+4+4a)j + (2i+5+4a)k}}{\qp[j]{2} \qp[k]{2}}
=\sum_{M\ge 0}\sums{j,k\ge0 \\ j+k=M}\frac{q^{2\binom{k}{2} + k+ 8\binom{M}{2} + (4+2i+4a)M}}{\qp[j]{2} \qp[k]{2}}.
\end{align*}  
By Lemma \ref{AUXLEM} \eqref{eq:ms:fold:2,2:1,0:2,0,0,1,0:1,1} (with $q$ replaced by $-q$), we see
\eqref{eq:l7:prf:1} is reduced to
\begin{align}
\sum_{i,M\ge0}\frac{q^{\binom{i}{2} + i + 2iM + 4M^2+4aM} (-q;q^2)_M}{(q;q)_i (q^2;q^2)_{M}}.\label{eq:l7:prf:2}
\end{align}
With (\eqEulerr), the inner sum on $i$ is rewritten as
\begin{align*}
\sum_{i\ge0} \frac{q^{\binom{i}{2} + (1+2M)i}}{(q;q)_i}
= (-q^{1+2M};q)_\infty
= \frac{(-q;q)_\infty}{(-q;q)_{2M}}
= \frac{1}{(q;q^2)_\infty (-q;q)_{2M}}.
\end{align*}
By \eqref{RRidentities}, we see $\displaystyle\eqref{eq:l7:prf:2}\cdot(q;q^2)_\infty$ is reduced to
\begin{align*}
\sum_{M\ge0}\frac{q^{4M^2+4aM} (-q;q^2)_M}{(q^2;q^2)_{M} (-q;q)_{2M}}
=\sum_{M\ge0}\frac{q^{4M^2+4aM}}{(q^4;q^4)_{M}}
=\frac{1}{(q^{4+4a},q^{16-4a};q^{20})_\infty}.
\end{align*}
This is equal to $(q;q^2)_\infty\cdot\chi_{A^{(2)}_2}((5-4a)\Lambda_0+(1+2a)\Lambda_1)$ and proves Theorem \ref{theo:A22:7:1}.

\begin{Rem}\label{otherslater}
In Slater's list~\cite[(79)=(98),(96)]{Sla}, 
there are identities whose
infinite products matches those in Theorem \ref{theo:A22:7:1} (but we do not need them).
\begin{align*}
\sum_{n\ge0} \frac{q^{n^2}}{(q;q)_{2n}} = \princhar{A^{(2)}_{2}}{5\Lambda_0+\Lambda_1},\quad 
\sum_{n\ge0} \frac{q^{n^2+2n}}{(q;q)_{2n+1}} = \princhar{A^{(2)}_{2}}{\Lambda_0+3\Lambda_1}.
\end{align*}
\end{Rem}

\section{Proof of Theorem \ref{theo:A22:7:2}}\label{pr:A22:l7:2}
Note that the Andrews-Gordon type series is of the form
\begin{align}
\sum_{i,j,k,\ell\ge0}(-1)^{k+(1-a)(j+k)}
\frac{q^{\quadddd{}{2}{2}{8} + \inteeee{}{}{2}{4}{4}{4} + \lineeee{}{(1+a)}{(3+a)}{(6+2a)}}}{\denommmm{}{2}{2}{4}}
\label{eq:l7:2:prf:1}
\end{align}
for $a=0,1$ (swapping $j$ and $k$ from the expression in Theorem \ref{theo:A22:7:2} when $a=0$).
We rewrite the inner sum on $j,k$ and $\ell$ in \eqref{eq:l7:2:prf:1} as
\begin{align*}
{} &{}
\sum_{j,k,\ell\ge0}(-1)^{k+(1-a)(j+k)}
\frac{q^{2\binom{j}{2} + 2\binom{k}{2} + 8\binom{\ell}{2}+ 4(jk +j\ell +k\ell)+ (1+a+i)j + (3+a+i)k + (6+2a+2i)\ell}}{\qp[j]{2} \qp[k]{2} \qp[\ell]{4}} \\
&=\sum_{M\ge0}\sums{j,k,\ell\ge0 \\ j+k+2\ell=M}(-1)^{k+(1-a)(j+k+2\ell)}
\frac{q^{2jk+ 2k + 2\ell+ 2\binom{j+k+2\ell}{2}+ (1+a+i)(j+k+2\ell)}}{\qp[j]{2} \qp[k]{2} \qp[\ell]{4}}
\end{align*}
By Lemma \ref{AUXLEM} \eqref{eq:ms:fold:112:010:0,0,0;1,0,0;0,1,1:112} (with $q$ replaced by $q^2$), we see
\eqref{eq:l7:2:prf:1} is reduced to
\begin{align}
\sum_{i,M\ge0} (-1)^{(1-a)M}\frac{q^{\binom{i}{2} + i + 2\binom{M}{2} + (1+a+i)M} (-q^2;q^2)_{[M/2]}}{(q;q)_i (q^2;q^2)_{[M/2]}}.
\label{eq:l7:2:prf:2}
\end{align}
With (\eqEulerr), the inner sum on $i$ is rewritten as
\begin{align*}
\sum_{i\ge0} 
\frac{q^{\binom{i}{2} + (1+M)i}}{(q;q)_i}
= (-q^{1+M};q)_{\infty}
= \frac{(-q;q)_\infty}{(-q;q)_M}
= \frac{1}{(q;q^2)_\infty (-q;q)_M}.
\end{align*}
Hence,  \eqref{eq:l7:2:prf:2}$\cdot(q;q^2)_\infty$ is equal to
\begin{align}
\sum_{M\ge0} (-1)^{(1-a)M}\frac{q^{2\binom{M}{2} + (1+a)M} (-q^2;q^2)_{[M/2]}}{(-q;q)_M (q^2;q^2)_{[M/2]}}
= \sum_{M\ge0} (-1)^{(1-a)M}\frac{q^{2\binom{M}{2} + (1+a)M}}{(-q;-q)_M}.
\label{eq:l7:2:prf:33}
\end{align}
Further, by Lemma \ref{AUXLEM} \eqref{eq:ms:fold:22:00:0,0,0,0,1:11} (with $q$ replaced by $-q$), we see
\eqref{eq:l7:2:prf:33} is equal to
\begin{align}
\sum_{m,n\ge0} (-1)^{(1-a)(m+n) + n} \frac{q^{2\binom{m+n}{2} + (1+a)(m+n) + n}}{\qp[m]{2} \qp[n]{2}}.
\label{eq:l7:2:prf:3}
\end{align}

Finally, by (\eqEulerr), \eqref{eq:l7:2:prf:3} is reduced to
\begin{align*}
\begin{cases}
\displaystyle\sum_{n\ge0} \frac{q^{2\binom{n}{2} + 2n}}{\qp[n]{2}}\sum_{m\ge0} (-1)^{m}\frac{q^{2\binom{m}{2} + 2mn + m}}{\qp[m]{2}}
= \sum_{n\ge0} \frac{q^{2\binom{n}{2} + 2n}}{\qp[n]{2}}(q^{2n+1};q^2)_\infty & (a=0),\\
\displaystyle\sum_{m\ge0} \frac{q^{2\binom{m}{2} + 2m}}{\qp[m]{2}}\sum_{n\ge0} (-1)^{n}\frac{q^{2\binom{n}{2} + 2mn + 3n}}{\qp[n]{2}} 
= \sum_{m\ge0} \frac{q^{2\binom{m}{2} + 2m}}{\qp[m]{2}}(q^{2m+3};q^2)_\infty & (a=1).
\end{cases}
\end{align*}

In each case, we see $\eqref{eq:l7:2:prf:1}=\eqref{eq:l7:2:prf:3}/(q;q^2)_\infty$ is equal to 
\begin{align*}
\sum_{s\ge0} \frac{q^{2\binom{s}{2} + 2s}}{\qp[s]{2}}\frac{(q^{2s+1+2a};q^2)_\infty}{(q;q^2)_\infty}
=\sum_{s\ge0} \frac{q^{2\binom{s}{2} + 2s}}{\qp[s]{2}(q;q^2)_{s+a}}
=\sum_{s\ge0} \frac{q^{s^2+s}}{(q;q)_{2s+a}}.
\end{align*}
This proves Theorem \ref{theo:A22:7:2} in virtue of Theorem \ref{SLATERTHM}.

\section{Conjectures for level 2 modules of $A_{13}^{(2)}$}\label{conjs}
The level 2 principal characters of $A^{(2)}_{13}$ are 
\begin{align*}
\princhar{A^{(2)}_{13}}{(\delta_{i0}+\delta_{i1})\Lambda_0+\Lambda_i}
= \frac{(q^{16};q^{16})_{\infty}}{(q^2;q^{2})_\infty}\frac{[q^{2i};q^{16}]_\infty}{[q^i;q^{16}]_\infty},
\end{align*}
where $0\leq i\leq 7$ and $\delta$ is the Kronecker delta.

We give conjectural Andrews-Gordon type series for 
$\princhar{A^{(2)}_{13}}{\Lambda_0+\Lambda_1}$ and $\princhar{A^{(2)}_{13}}{\Lambda_{2n+1}}$, where $n=1,2,3$. 
Note that $\princhar{A^{(2)}_{13}}{2\Lambda_0}=\princhar{A^{(2)}_{13}}{2\Lambda_1}$ and $\princhar{A^{(2)}_{13}}{\Lambda_{2n}}$, where $n=1,2,3$,  
are obtained by inflating $q$ to $q^2$ from infinite products with smaller period.

\begin{Conj}\label{conj:A13:3sum:1}
We have $F_{\numANiJusanIchi}(2,2,2) = \princhar{A^{(2)}_{13}}{\Lambda_3}$,
$F_{\numANiJusanIchi}(4,2,6) = \princhar{A^{(2)}_{13}}{\Lambda_5}$ and
$F_{\numANiJusanIchi}(6,4,6) = \princhar{A^{(2)}_{13}}{\Lambda_7}$, where
\begin{align*}
F_{\numANiJusanIchi}(a,b,c):=
\sum_{i,j,k\ge0} (-1)^k \frac{q^{\quaddd{4}{2}{4} + \inteee{2}{4}{4} + \lineee{a}{b}{c}}}{\qp[i]{} \qp[j]{2} \qp[k]{4}}.
\end{align*}
\end{Conj}

\begin{Conj}\label{conj:A13:3sum:2}
We have $F_{\numANiJusanNi}(1,3,12) = \princhar{A^{(2)}_{13}}{\Lambda_0+\Lambda_1}$,
$F_{\numANiJusanNi}(1,1,8) = \princhar{A^{(2)}_{13}}{\Lambda_3}$ and
$F_{\numANiJusanNi}(3,3,16) = \princhar{A^{(2)}_{13}}{\Lambda_7}$, where
\begin{align}
\label{eq:A13:3sum:2}
F_{\numANiJusanNi}(a,b,c):=
\sum_{i,j,k\ge0} (-1)^j \frac{q^{\quaddd{}{2}{16} + \inteee{2}{4}{4} + \lineee{a}{b}{c}}}{\qp[i]{} \qp[j]{2} \qp[k]{4}}.
\end{align}
\end{Conj}

\begin{Conj}\label{conj:A13:4sum:1}
We have $F_{\numANiJusanSan}(1,5,1,12) = \princhar{A^{(2)}_{13}}{\Lambda_5}$, where
\begin{align*}
F_{\numANiJusanSan}(a,b,c,d):=
\sum_{i,j,k,\ell\ge0} (-1)^k \frac{q^{\quadddd{2}{4}{2}{16} + \inteeee{2}{2}{4}{4}{8}{4} + \lineeee{a}{b}{c}{d}}}{\qp[i]{} \qp[j]{2} \qp[k]{2} \qp[\l]{4}}.
\end{align*}
\end{Conj}

\begin{Rem}
One can prove 
$F_{\numANiJusanNi}(a,b,c) = F_{\numANiJusanSan}(a,2a+1,b,c)$ for $a,b,c\geq 0$
by rewriting the inner sum on $i$ in \eqref{eq:A13:3sum:2} as
\begin{align*}
\sum_{i\ge0} \frac{q^{\binom{i}{2} + (2j+4k+a)i}}{\qp[i]{}}
&= \sum_{i\ge0} \sums{s,t\ge0 \\ s+2t=i}\frac{q^{\binom{s+2t}{2} + (2j+4k+a)(s+2t) + \binom{s}{2}}}{\qp[s]{} \qp[t]{2}}\\
&= \sum_{s,t\ge0} \frac{q^{2\binom{s}{2} + 4\binom{t}{2} + 2st + (2j+4k)(s+2t) + as + (2a+1)t}}{\qp[s]{} \qp[t]{2}}.        
\end{align*}
Here, the first equality follows from
$\displaystyle\sum_{\substack{i,j\geq 0 \\ i+2j=M}} \frac{q^{\binom{i}{2}}}{\qp[i]{} \qp[j]{2}} = \frac{1}{\qp[M]{}}$ for $M\geq0$,
which is proved similarly to Lemma \ref{AUXLEM} \eqref{eq:ms:fold:22:00:0,0,0,0,1:11} (using (\eqEuler) and (\eqEulerr)).
\end{Rem}

Hence, if Conjecture \ref{conj:A13:3sum:2} is true, we have
$F_{\numANiJusanSan}(1,3,3,12) = \princhar{A^{(2)}_{13}}{\Lambda_0+\Lambda_1},$
$F_{\numANiJusanSan}(1,3,1,8) = \princhar{A^{(2)}_{13}}{\Lambda_3}$,
$F_{\numANiJusanSan}(3,7,3,16) = \princhar{A^{(2)}_{13}}{\Lambda_7}$
and Conjecture \ref{conj:A13:4sum:1} gives the ``missing'' case.

\begin{Rem}\label{recstrr}
The double sums \eqref{eq:AG:234:14} and \eqref{eq:AG:146:14} coincide with
those obtained by taking the ``$k=0$ part'' of the triple sums 
$F_\numANiJusanNi(1,1,8)$ and $F_\numANiJusanNi(1,3,12)$ in Conjecture \ref{conj:A13:3sum:2}.
\end{Rem}


\section{Notes on Capparelli's identities}\label{Capsum}
We fix the conditions \textsf{(C1)} and \textsf{(C2)} on a partition $\lambda=(\lambda_1,\cdots,\lambda_{\ell})$
to recall Capparelli's partition theorems (Theorem \ref{CAPTHM}). See also \S\ref{Atwotwo}. 
\begin{enumerate}
\item[\textsf{(C1)}] $1\leq\forall j\leq \ell-1$, $\lambda_j-\lambda_{j+1}\geq 2$, 
\item[\textsf{(C2)}] $1\leq\forall j\leq \ell-1$, $\lambda_j-\lambda_{j+1}\leq 3\implies \lambda_j+\lambda_{j+1}\equiv 0 \pmod 3$
\end{enumerate}

\begin{Thm}[{\cite{An4,Cap1,TX}}]\label{CAPTHM}
Let $a=1,2$.
For any $n\geq 0$, partitions $\lambda$ of $n$ with
condtion $C_a$ are equinumerous to those with condition $D_a$, where
\begin{enumerate}
\item[$C_a$:] \textup{\textsf{(C1), (C2)}} and $1\leq\forall j\leq \ell,\lambda_{i}\neq a$,
\item[$D_a$:] $1\leq\forall j\leq \ell,\lambda_j\not\equiv \pm a\pmod{6}$, and $\lambda_1,\dots,\lambda_{\ell(\lambda)}$ are distinct.
\end{enumerate}
\end{Thm}

In \cite[Theorem 10, Theorem 11]{Ku2}, Kur\c{s}ung\"oz showed
\begin{align}
f_{1}(x,q) &= \sum_{i,j\ge0} \frac{q^{\quadd{4}{12}+\intee{6}+\linee{2}{6}}}{\qp[i]{} \qp[j]{3}} x^{i+2j}\label{eq:Cap_2sum1}, \\
f_{2}(x,q) &= \sum_{i,j\ge0} \frac{q^{\quadd{4}{12}+\intee{6}+\linee{3}{9}}}{\qp[i]{} \qp[j]{3}} x^{i+2j} (1+xq^{1+2i+3j})\label{eq:Cap_2sum2},
\end{align}
where $\mathcal{C}_{a}$ denote the set of partitions with the condition $C_a$ 
and $f_{a}(x,q) := \sum_{\lambda\in\mathcal{C}_a} x^{\ell(\lambda)} q^{|\lambda|}$
for $a=1,2$. 
Combining Theorem \ref{CAPTHM} and
\eqref{eq:Cap_2sum1},\eqref{eq:Cap_2sum2} with $x=1$, Kur\c{s}ung\"oz got the following identities.

\begin{Thm}[{\cite[Corollary 18]{Ku2}}]\label{KUres}
Concerning the
level 3 modules of $A^{(2)}_2$, we have
\begin{align*}
\sum_{i,j\geq 0}\frac{q^{2i^2+6ij+6j^2}}{(q;q)_i(q^3;q^3)_j}
&=(-q^2,-q^3,-q^4,-q^6;q^6)_{\infty}\bigg(=\frac{1}{[q^2,q^3;q^{12}]_\infty}=\chi_{A^{(2)}_2}(3\Lambda_0)\bigg), \\
\sum_{i,j\geq 0}\frac{q^{2i^2+6ij+6j^2+i+3j}(1+q^{2i+3j+1})}{(q;q)_i(q^3;q^3)_j}
&=(-q,-q^3,-q^5,-q^6;q^6)_{\infty}\\\bigg(&=\frac{[q^2;q^{12}]_\infty}{[q,q^3,q^5;q^{12}]_\infty}=\chi_{A^{(2)}_2}(\Lambda_0+\Lambda_1)\bigg).
\end{align*}
\end{Thm}

Note that the left hand side of the latter is not
an Andrews-Gordon type series in our sense (see
\S\ref{mainres}). The purpose of this section 
is to prove 
\begin{align*}
\sum_{i,j,k\ge0}\frac{q^{\quaddd{5}{5}{12}+\inteee{3}{6}{6}+\lineee{(3-a)}{(2+a)}{6}}}{\denommm{2}{2}{3}} 
=\chi_{A^{(2)}_2}((5-2a)\Lambda_0+(a-1)\Lambda_1),
\end{align*}
for $a=1,2$. This follows from substituting $x=1$ to Theorem \ref{KUresss}.

\begin{Thm}\label{KUresss} For $a=1,2$, we have
\begin{align*}
f_{a}(x,q) = \sum_{i,j,k\ge0}\frac{q^{\quaddd{5}{5}{12}+\inteee{3}{6}{6}+\lineee{(3-a)}{(2+a)}{6}}}{\denommm{2}{2}{3}} x^{i+j+2k}.
\end{align*}
\end{Thm}

\begin{proof}
Since the set of partitions with the conditions \textsf{(C1)},\textsf{(C2)} is a 
\emph{linked partition ideal} (see ~\cite[\S8]{An1}),
one can derive a $q$-difference equation algorithmically
\begin{align}\label{eq:qd:F}
F(x) = (1+xq^3)F(xq^3) + x(q^{3-a} + q^{3+a} + xq^6)F(xq^6) + x^2q^9(1-xq^6)F(xq^9),
\end{align}
where $F(x):=f_{a}(x,q)$. 
Putting $F(x)=:\sum_{M\ge0}f_M x^M$, by \eqref{eq:qd:F} we have
\begin{align*} 
  (1-q^{3M})f_{M}
  = q^{3M-3} (q^{3M+a} + q^{3M-a} + q^3) f_{M-1}
  + q^{6M-6} (1 + q^{3M-3}) f_{M-2}
  - q^{9M-12} f_{M-3}
\end{align*}
for all $M\in\Z$ (we consider $f_M=0$ for $M<0$).
Putting $g_M:=q^{-3\binom{M}{2}}f_M$, we have
\begin{align*} 
(1-q^{3M})g_{M}
= (q^{3M+a} + q^{3M-a} + q^3) g_{M-1}
  + (q^{3} + q^{3M}) g_{M-2} - q^{6} g_{M-3}.
\end{align*}
Putting $G(x):=\sum_{M\ge0}g_M x^M$,  we have
\begin{align*}
(1-xq^3)(1-x^2q^3)G(x) = (1+xq^{3-a})(1+xq^{3+a})G(xq^3),
\end{align*}
and hence
\begin{align}\label{eq:sol:G}
G(x) 
= \frac{(-xq^{3-a},-xq^{3+a};q^3)_\infty}{(xq^3;q^3)_\infty (x^2q^3;q^6)_\infty}
= \frac{(-xq^{3-a},-xq^{2+a};q^2)_\infty}{(x^2q^3;q^3)_\infty},
\end{align}
where the latter equality is because $a=1,2$. 
By (\eqEuler) and (\eqEulerr) (in \S\ref{prep}) we have
\begin{align*}
G(x) 
= \sum_{i,j,k\ge0} 
\frac{q^{2\binom{i}{2}+(3-a)i}}{\qp[i]{2}}
\frac{q^{2\binom{j}{2}+(2+a)j}}{\qp[j]{2}}
\frac{q^{3k}}{\qp[k]{3}}
x^{i+j+2k}.
\end{align*}
Finally, since $f_M:=q^{3\binom{M}{2}}g_M$ we have 
\begin{align*}
F(x) 
= \sum_{i,j,k\ge0} 
\frac{q^{2\binom{i}{2}+(3-a)i}}{\qp[i]{2}}
\frac{q^{2\binom{j}{2}+(2+a)j}}{\qp[j]{2}}
\frac{q^{3k}}{\qp[k]{3}}
q^{3\binom{i+j+2k}{2}}
x^{i+j+2k},
\end{align*}
which is precisely Theorem \ref{KUresss}.
\end{proof}

\begin{Rem}
In place of \eqref{eq:sol:G}, if we write 
\begin{align*}
G(x) = \frac{(-xq^{2};q)_\infty}{(x^2q^3;q^3)_\infty} \ \text{for $a=1$ and} \
G(x) = (1+xq)\frac{(-xq^{3};q)_\infty}{(x^2q^3;q^3)_\infty} \ \text{ for $a=2$},
\end{align*}
then we get the double sum expression \eqref{eq:Cap_2sum1} 
and an alternative one to \eqref{eq:Cap_2sum2}
\begin{align*}
f_{2}(x,q) = \sum_{i,j\ge0} \frac{q^{\quadd{4}{12}+\intee{6}+\linee{3}{6}}}{\qp[i]{} \qp[j]{3}} x^{i+2j} (1+xq^{1+3i+6j}).
\end{align*}
\end{Rem}

\begin{Rem}
We can reprove Theorem \ref{CAPTHM}
using the equation \eqref{eq:qd:F}.
If we 
put $G(x)=\sum_{M\ge0}g_Mx^M:=F(x)/(x;q^3)_\infty$,
$h_M:=g_M/(-q^3;q^3)_M$ and $H(x):=\sum_{M\ge0} h_M x^M$,
by a similar argument to the proof of Proposition \ref{KUresss}, we get
\begin{align*}
(1-x)(1-xq^3)H(x) = (1+xq^{3-a})(1+xq^{3+a})H(xq^6)
\end{align*}
and $H(x)=(-xq^{3-a},-xq^{3+a};q^6)_\infty/(x;q^3)_\infty$.
Again, by similar arguments (using (\eqEuler) and (\eqEulerr)), we see
\begin{align*}
g_M
= \sum_{i+j+k=M}
\frac{1}{(q^3;q^3)_i}
\frac{q^{6\binom{j}{2}+(3-a)j}}{(q^6;q^6)_j}
\frac{q^{6\binom{k}{2}+(3+a)k}}{(q^6;q^6)_k}
(-q^3;q^3)_{M}.
\end{align*}
Since $F(x)=G(x)(x;q^3)_\infty$,
by Appell's Comparison Theorem \cite[page 101]{Die} we get 
\begin{align*}
F(1) 
&= (q^3;q^3)_\infty \lim_{M\rightarrow\infty}g_M
= (-q^3;q^3)_\infty
\sum_{j,k\ge 0}\frac{q^{6\binom{j}{2}+(3-a)j}}{(q^6;q^6)_j}\frac{q^{6\binom{k}{2}+(3+a)k}}{(q^6;q^6)_k} \\
&= (-q^3;q^3)_\infty (-q^{3-a},-q^{3+a};q^6)_\infty,
\end{align*}
which proves Theorem \ref{CAPTHM}.
\end{Rem}


\end{document}